\documentclass[11pt, reqno]{article}
\usepackage{color}

\usepackage{amsmath}
\usepackage{amssymb}
\usepackage{amsfonts}
\newenvironment{proof}[1][Proof]{\noindent\textit{#1}\quad }
{\hfill $\Box$\vspace{0.7mm}}
\usepackage[lmargin=3cm,rmargin=3cm,tmargin=3cm,bmargin=3cm]{geometry}

\def\R{\mathbb{R}}

\numberwithin{equation}{section}

\newtheorem{theo}{Theorem}[section]
\newtheorem{prop}[theo]{Proposition}
\newtheorem{lem}[theo]{Lemma}
\newtheorem{rem}[theo]{Remark}
\newtheorem{rems}[theo]{Remarks}
\newtheorem{defi}[theo]{Definition}
\newtheorem{cor}[theo]{Corollary}
\numberwithin{equation}{section}

\newcommand{\beq}{\begin{equation}}
\newcommand{\eeq}{\end{equation}}

\newcommand{\bean}{\begin{eqnarray}}
\newcommand{\eean}{\end{eqnarray}}

\newcommand{\bea}{\begin{eqnarray*}}
\newcommand{\eea}{\end{eqnarray*}}

\newcommand{\bd}{\begin{description}}
\newcommand{\ed}{\end{description}}

\newcommand{\bc}{\begin{center}}
\newcommand{\ec}{\end{center}}

\newcommand{\ben}{\begin{enumerate}}
\newcommand{\een}{\end{enumerate}}

\newcommand{\bit}{\begin{itemize}}
\newcommand{\eit}{\end{itemize}}

plus 2pt minus 2pt

\begin{document}
\title{A 2D Schr\"odinger equation with time-oscillating exponential nonlinearity}
\date{}
\maketitle

{\small
\begin{center}
{\sc Abdelwahab Bensouilah}\\
Laboratoire Paul Painlev\'e (U.M.R. CNRS 8524), U.F.R. de Math\'ematiques,\\ Universit\'e Lille 1, 59655 Villeneuve d'Ascq Cedex, France.\\
E-mail address: ai.bensouilah@math.univ-lille1.fr\\
{\sc Dhouha Draouil}\\
Universit\'e de Tunis El Manar, Facult\'e des Sciences de Tunis, D\'epartement de Math\'ematiques, Laboratoire \'equations aux d\'eriv\'ees partielles (LR03ES04), 2092 Tunis, Tunisie.\\
E-mail address:  douhadraouil@yahoo.fr \\
 {\sc Mohamed Majdoub} \\
Department of Mathematics, College of Science, Imam Abdulrahman Bin Faisal University,\\
P.O. Box 1982, Dammam, Saudi Arabia.\\
E-mail address:  mmajdoub@iau.edu.sa
\end{center}
}
\numberwithin{equation}{section}
\allowdisplaybreaks

 \begin{quote}
\footnotesize
{\bf Abstract.}
This paper deals with the 2-D Schr\"odinger equation with time-oscillating exponential nonlinearity
$i\partial_t u+\Delta u= \theta(\omega t)\big(e^{4\pi|u|^2}-1\big)$, where $\theta$ is a periodic  $C^1$-function. We prove that for a class of initial data $u_0 \in H^1(\R^2)$, the solution $u_{\omega}$ converges, as $|\omega|$ tends to infinity to the solution $U$ of the limiting equation $i\partial_t U+\Delta U= I(\theta)\big(e^{4\pi|U|^2}-1\big)$ with the same initial data, where $I(\theta)$ is the average of $\theta$.
\end{quote}



\noindent{\bf MR(2010) Subject Classification}\;\;\;{35-xx, 35Q55}
\maketitle


\eject
\section{Introduction}

Recall the monomial defocusing semilinear Schr\"odinger equation in space dimension
 $N\geq 1$
\begin{equation}
\label{NLSp} i\, \partial_t u + \Delta u =  |u|^{p-1} u,\quad u: \R^{1+N}
\longrightarrow \mathbb{C},
\end{equation}
which has the critical exponents $p^*=\frac{N+2}{N-2}$ (for $N\ge 3$) and $p_*=1+\frac{4}{N} $.

For the {\it{energy subcritical}} case ($p < p^*$), an iteration
of the local-in-time well-posedness result using the {\it{a priori}}
upper bound on $\| u(t) \|_{H^1}$ implied by the conservation laws
establishes global well-posedness for \eqref{NLSp} in $H^1$. Those
solutions scatter when $p>p_*$ ( see \cite{GV4,2Dsubcrit}).

The {\it{energy critical}} case ($p=p^*$) is actually harder than the Klein-Gordon (wave) equation, for which the finite propagation property was crucial to exclude possible concentration of energy, whereas there is no upper bound on the propagation speed for the Schr\"odinger equation.
Nevertheless, based on new ideas such as induction on the energy size and frequency split propagation estimates, Bourgain in  \cite{Bourgain1} proved
global well-posedness and scattering for radially symmetric data, and this result was extended to the general case by Colliander et al. in \cite{CKSTT} using a new interaction Morawetz  inequality.

For $N=2$, the initial value problem \eqref{NLSp} is energy subcritical for
all $p > 1$. To identify an "energy critical" nonlinear
Schr\"odinger initial value problem on $\R^2$, so, it is natural
to consider problems with exponential nonlinearities. According to the sharp Trudinger-Moser inequality on $\R^2$ \cite{AT, Ru} and the 2D critical Sobolev embedding \cite{JFA}, it is natural  to investigate the following Cauchy problem
\begin{equation}
\label{NLSexp}
\left\{
\begin{matrix}
i\partial_t u+\Delta u= u\big(e^{4\pi|u|^2}-1\big), & u:\R^{1+2 }\longrightarrow \mathbb{C} ,\\
u(0) = u_0\in H^1 (\R^2)\,. \\
\end{matrix}
\right.
\end{equation}
Solutions of (\ref{NLSexp}) formally satisfy the conservation of mass and Hamiltonian

\begin{equation}
\label{mass}
M(u(t)):=\|u(t)\|_{L^2}^2=M(u(0)),
\end{equation}

\begin{eqnarray}
\label{ham} H(u(t))&:=&\Big\|\nabla u(t)\Big\|_{L^2}^2+
\frac{1}{4\pi}\Big\|e^{4\pi |u(t)|^2}-1-4\pi|u(t)|^2\Big\|_{L^1(\R^2)}\\
\nonumber &=&H(u(0)).
\end{eqnarray}
For a such problem, global well-posedness together with the scattering for small data were obtained in \cite{NO3}. Using the sharp Trudinger-Moser inequality on $\R^2$, the size of the initial data for which one has local existence was quantified in \cite{JHDE}, and a notion of criticality was proposed:
\begin{defi}
 The Cauchy problem \eqref{NLSexp} is said to be {\it subcritical} if
$H(u_0)<1$, {\it critical} if $H(u_0)=1$ and {\it supercritical} if $H(u_0)>1$.
\end{defi}
The reason behind this definition lies in the fact that one can construct a unique local solution for initial data $u_0$ such that $\|\nabla
 u_0\|_{L^2}<1$, and the time of existence depends only on $\eta:=1-\|\nabla u_0\|_{L^2}$
and $\| u_0\|_{ L^2} $. Therefore the maximal  solution is global in the subcritical case, while in the critical case a concentration phenomena of the Hamiltonian may happens. The following global well-posedness result was proved in \cite{JHDE}.
\begin{theo}
\label{GWP-NLS}
 Assume that $H(u_0)\le 1$, then the problem \eqref{NLSexp}
has a unique global solution $u$ in the class
$$
{\mathcal C}(\R, H^1(\R^2)).
$$
Moreover, $u\in L^4_{loc}(\R,\;{\mathcal C}^{1/2}(\R^2))$ and
satisfies the conservation of the mass and the Hamiltonian.
\end{theo}

In the subcritical case, a scattering result was obtained in \cite{Nonlinearity} where the cubic term was subtracted from the non linearity to avoid the critical value $p_*=1+{4\over N}$. More precisely
\begin{theo}
 \label{Main-NLS}
For any global solution $u$ of \eqref{NLSexp} in $H^1$ satisfying
$H(u)<1$, we have $u\in L^4(\R, {\mathcal C}^{1/2})$ and
there exist unique free solutions $u_{\pm}$ such that
$$
\|(u-u_{\pm})(t)\|_{H^1}\to 0\qquad (t\to\pm\infty).
$$
Moreover, the maps
$$
u(0)\longmapsto u_{\pm}(0)
$$
are homeomorphisms between the unit balls in the nonlinear energy
space and the free energy space, namely from $\{\varphi\in H^1\ ; \ H(\varphi)< 1\}$ onto $\{\varphi\in H^1\ ;\  \|\nabla\varphi\|_{L^2}<1\}$.
 \end{theo}
The main ingredient for the subcritical case is a new interaction Morawetz
estimate, proved independently by Colliander et al. and Planchon-Vega
\cite{CGT, PV}.
\begin{rems}\quad\\{\rm
\vspace{-0.8cm}
\bit
\item[i)] The proof in the subcritical case is much simpler for {\tt NLS} than {\tt NLKG} \cite{Duke}, given the a priori estimate due to \cite{CGT,PV}.
\item[ii)] This result was extended in \cite{BIP} to the critical case, but only in the radial framework.
\eit}
\end{rems}


\subsection{Setting of the Problem and Main Results}


In some recent works \cite{Cazenave, Fang}, the following initial value problem was investigated:
\begin{equation}
\label{NLSOsci}
\left\{
\begin{matrix}
i\partial_t u+\Delta u+\theta(\omega t)|u|^{\alpha}u=0, \\
u(0) = \varphi\in H^1 (\R^N), \\
\end{matrix}
\right.
\end{equation}
where $\theta\in C^1(\R,\R)$ is a $\tau$-periodic  function for some $\tau > 0$, $\omega\in\R$ and $\alpha\leq \frac{4}{N-2}$ ($N\geq 3$). A typical example is $\theta(s)=\lambda_0+\lambda_1\sin( s)$ with $\lambda_0, \lambda_1\in\R$. It is shown in \cite{Cazenave, Fang} that the solution $u_\omega$ converges as $|\omega|\to\infty$ to the
solution $U$ of the limiting equation $i\partial_t U +\Delta U + I(\theta)|U|^\alpha U=0$ with the same initial
condition, where $I(\theta)$ is the average of $\theta$ given by
 \beq
 \label{aver}
 I(\theta)={1\over \tau}\int_0^\tau\,\theta(s)\,ds\,.
 \eeq

 It is the aim of this note to extend the results of \cite{Cazenave, Fang} to the 2-D critical semilinear Schr\"odinger equation. Thus we consider the initial value problem
\begin{equation}
\label{MainP}
\left\{
\begin{matrix}
i\partial_t u+\Delta u= \theta(\omega t) u\bigg(e^{4\pi|u|^2}-1\bigg); \\
u(0) = u_0 \in H^1 (\R^2),\\
\end{matrix}
\right.
\end{equation}
where $\omega\in\R$ and $\theta : \R\to\R$ is a $C^1$-function satisfying
\beq
\label{teta1}
\theta\quad\mbox{is}\quad \tau-\mbox{periodic for some}\quad \tau>0;
\eeq
\beq
\label{teta2}
I(\theta)\geq 0.
\eeq

The equivalent integral form of \eqref{MainP} reads as follows
\begin{eqnarray}
\label{T}
u(t)=e^{it\Delta}u_0+i\int_0^te^{i(t-s)\Delta}\theta (\omega s)u(s) \left(e^{4\pi|u(s)|^2}-1\right)ds,
\end{eqnarray}
where $\bigg(e^{it\Delta}\bigg)_{t\in\R}$ is the Schr\"odinger group. Solutions to \eqref{MainP} formally satisfy the conservation of mass.

Remarking that the function $\theta$ is  uniformly bounded, we only take its $L^\infty$-norm when estimating the nonlinearity. Hence, using similar arguments as in \cite{JHDE}, we can prove local well-posedness of \eqref{MainP} in the energy space.
\begin{prop}
\label{LWP-omega}
For every $u_0\in H^1(\R^2)$ such that $\| \nabla u_0\|_{L^2}<1$, there exists a unique maximal $H^1$-solution $u_\omega\in C((-T_{*},T^{*}); H^1)$ to \eqref{MainP} with $0<T_{*},T^{*}\leqslant \infty$. Moreover, $u_{\omega}\in L_{loc}^q((-T_{*},T^{*}),W^{1,r}(\R^2))$ for all admissible pairs $(q,r)$ $($see \eqref{admiss}$)$.
\end{prop}

Our main goal is to investigate the behavior of $u_\omega$ as $\arrowvert \omega\arrowvert\rightarrow +\infty$. It is natural to expect that $u_\omega$ behaves like the solution $U$ of the following Cauchy problem as $|\omega|$ goes to infinity.
\begin{equation}
\label{LP}
\left\{
\begin{matrix}
i\partial_t U+\Delta U =I(\theta)U\bigg(e^{4\pi\arrowvert U\arrowvert^2}-1\bigg);\\
U(0) = u_0 \in H^1(\R^2),\\
\end{matrix}
\right.
\end{equation}
or equivalently
\beq
\label{ILP}
U(t)=e^{it\Delta} u_0+iI(\theta)\int_0^te^{i(t-s)\Delta}U(s)\left(e^{4\pi\arrowvert U(s)\arrowvert^2}-1\right)ds.
\eeq
For an initial data $u_0\in H^1(\R^2)$ such that $\|\nabla u_0\|_{L^2}<1$, the Cauchy problem \eqref{LP} is locally well-posed and its maximal solution belongs to $C([0,S);H^1(\R^2))\bigcap L^q_{loc}((0,S);W^{1,r}(\R^2))$ for some $S>0$ and for all admissible pairs $(q,r)$. Moreover, the following conservation laws hold:
\beq \label{M}
M(U(t)):= \| U(t)\|_{L^2}^2 = M(u_0),
\eeq
and
\beq
\label{H}
  H(U(t)):= \Big\|\nabla U(t)\Big\|_{L^2}^2+
\frac{I(\theta)}{4\pi}\Big\|e^{4\pi |U(t)|^2}-1-4\pi|U(t)|^2\Big\|_{L^1(\R^2)}=H(u_0)
\eeq
Note that since $I(\theta)$ is positive, then for any initial data $u_0$ with $H(u_0)\leqslant 1$, the Cauchy problem \eqref{LP} is globally well-posed (see \cite{JHDE} for a proof). The main result of this paper reads.
\begin{theo}\label{EL}
Let $u_0 \in  H^1(\mathbb{R}^2)$ such that $H( u_0)<1$. Denote by $u_\omega \in C((-T_*,T^*); H^1)$ the maximal solution of \eqref{MainP} and $U  \in C(\R; H^1)$ the global solution of \eqref{LP}.
\bit
\item[i)] For any $0<T<\infty$, the solution $u_{\omega}$ exists on $[0,T]$ for  $|\omega|$ sufficiently large.
\item[ii)] Assume that for $0<T<\infty$, there exists a constant $0\leqslant A(T)<1$ such that
\beq
\label{AT}
\sup_{t \in [0,T]}\,\|\nabla u_{\omega}(t)\|_{L^2}\leqslant A(T),
\eeq
for $|\omega|$ sufficiently large. Then, $u_{\omega}\rightarrow U$ in $L^q((0,T);W^{1,r})$ as $|\omega|\rightarrow \infty$ for all admissible pairs $(q,r)$ and for any $0<T<\infty$. In particular, the convergence holds in $C([0,T];H^1(\R^2))$.
\eit
\end{theo}
\begin{rems}\quad\\{\rm
\vspace{-0.8cm}
\bit
\item[i)] Note that the solution $u_\omega$ of \eqref{MainP} is obtained by applying a fixed point argument as in \cite{JHDE}. It follows that the assumption \eqref{AT} holds at least for small $T$.
\item[ii)] Suppose that $I(\theta)<0$ and let $u_0\in H^1(\R^2)$ such that the solution $U$ of \eqref{LP} blows up in finite time (such initial data $u_0$ exists). We don't know whether or not the solution $u_\omega$ of \eqref{MainP} blows up in finite time for $|\omega|$ sufficiently large.
\item[iii)] The theorem does not say anything on what happens to the solution $u_\omega$ if the function $\theta$ changes its sign (note that, when $\theta$ is positive, its average $I(\theta)$ is also positive; so the latter fulfills the assumptions). In particular, the nature of solution $u_\omega$ (global or blowing-up) may change according to $\omega$ and $U(t=t_0)$.  {\it This will be considered in a forthcoming paper}.
\eit}
\end{rems}

The rest of the paper is organized as follows. Section 2 is devoted to give some useful tools needed in the proofs. In Section 3, we give some preliminary results which prepare the proof of our main theorem. The proof of Theorem \ref{EL} is done in Section 4. Finally, we state in the Appendix a Gronwall-type estimate used in the proof of Theorem \ref{EL}.
\section{Useful Tools}
In this section we collect some known and useful estimates.
\begin{prop}[\sf Moser-Trudinger inequality \cite{AT}]\quad\\
\label{MT}
Let $\alpha\in [0,4\pi)$. A constant $c_\alpha$ exists
such that
\begin{equation}
\label{MT1} \|\exp(\alpha |u|^2)-1\|_{L^1(\R^2)}\leq c_\alpha
\|u\|_{L^2(\R^2)}^2
\end{equation}
for all $u$ in $H^1(\R^2)$ such that $\|\nabla
u\|_{L^2(\R^2)}\leq1$. Moreover, if $\alpha\geq 4\pi$, then
(\ref{MT1}) is false.
\end{prop}
\begin{rem}
\label{rem} We point out that $\alpha=4\pi$ becomes admissible in
(\ref{MT1}) if we require $\|u\|_{H^1(\R^2)}\leq1$ rather than
$\|\nabla u\|_{L^2(\R^2)}\leq1$. Precisely, we have
\beq
\label{MT2}
\sup_{\|u\|_{H^1}\leq
1}\;\;\|\exp(4\pi |u|^2)-1\|_{L^1(\R^2)}<\infty
\eeq
 and this is
false for $\alpha>4\pi$. See \cite{Ru} for more details.
\end{rem}

The following estimate is an $L^\infty$ logarithmic inequality
which enables us to establish the link between $ \|e^{4\pi
|u|^2}-1\|_{L^1_T(L^2(\R^2))} $ and dispersion properties of
solutions of the linear Schr\"odinger equation.
\begin{prop}[\sf Log estimate \cite{IMM}]\quad\\
\label{LogEst}
Let $\beta\in]0,1[$. For any $\lambda>\frac{1}{2\pi\beta}$ and
any $0<\mu\leq1$, a constant $C_{\lambda}>0$ exists such
that,
for any function $u\in H^1(\R^2)\cap{\mathcal C}^\beta(\R^2)$, we have
\begin{equation}
\label{Log}
\|u\|^2_{L^\infty}\leq
\lambda\|u\|_\mu^2 \log\left(C_{\lambda} +
\left(\frac{8}{\mu}\right)^\beta\frac{\|u\|_{{\mathcal C}^{\beta}}}{\|u\|_\mu}\right),
\end{equation}
where we set
\begin{equation}
\label{Hmu}
\|u\|_\mu^2:=\|\nabla u\|_{L^2}^2+\mu^2\|u\|_{L^2}^2.
\end{equation}
\end{prop}
Recall that ${\mathcal C}^{\beta}(\R^2)$ denotes the space of
$\beta$-H\"older continuous functions endowed with the norm $$
\|u\|_{{\mathcal
C}^{\beta}(\R^2)}:=\|u\|_{L^\infty(\R^2)}+\sup_{x\neq
y}\frac{|u(x)-u(y)|}{|x-y|^{\beta}}. $$
We refer to  \cite{IMM} for the proof of this proposition and more
details. We just point out that the condition  $ \lambda>
\frac{1}{2\pi\beta}$ in (\ref{Log}) is optimal.

In order to establish an energy estimate, one has to consider the
nonlinearity as a source term in \eqref{MainP}, so we need to estimate
it in the $L^1_t(H^1_x)$ norm. To do so, we use (\ref{MT1}) combined
with the so-called Strichartz estimate.
\begin{prop}[\sf Strichartz estimates \cite{Ca}]\quad\\
\label{Strich} Let $v_0$ be a function in $H^1(\R^2)$ and $F\in
L^1(\R, H^1(\R^2))$. Denote by $v$ the solution of the
inhomogeneous linear Schr\"odinger problem
$$
i\partial_t v+\Delta v=F(t,x),\quad v(0)=v_0.
$$
Then, a constant $C$ exists such that for any $T>0$ and any
admissible pairs of Strichartz exponents $(q,r)$ i.e
\beq
\label{admiss}
0\leq\frac{2}{q}=1-\frac{2}{r}<1,
\eeq
yields
\begin{equation}
\label{stri} \|v\|_{L^q([0,T],{W}^{1,r}(\R^2))}\leq
C\bigg[\|v_0\|_{H^1(\R^2)}+\| F\|_{L^1([0,T],H^1(\R^2))}\bigg].
\end{equation}
\end{prop}
In particular, note that $(q,r)=(4,4)$ is an admissible Strichartz
pairs and $${W}^{1,4}(\R^2)\hookrightarrow {\mathcal C}^{1/2}(\R^2).$$

\section{Preliminary Results}

In order to prove Theorem \ref{EL}, we need the next lemma
\begin{lem}
\label{main lem}
Fix an initial value $u_0\in H^1(\R^2)$ with $H( u_0)<1$. Given $\omega\in \R$, denote by $u_\omega$ the maximal solution of \eqref{MainP}.
Let $U$ be the unique global solution of \eqref{LP}. Fix $0<l<\infty$ and suppose also that $u_\omega$ satisfies
\beq
\label{A1}
\underset{|\omega| \to \infty}{\mbox{lim sup}}\|u_{\omega}\|_{L^4((0,l), C^{\frac{1}{2}}(\R^2))}:=\underset{\xi \to \infty}{\mbox{lim}} \, \left(\underset{|\omega| \geqslant \xi}{\mbox{sup}} \|u_{\omega}\|_{L^4((0,l), C^{\frac{1}{2}}(\R^2))}\right) <\infty,
\eeq
and, for $|\omega|$ sufficiently large
\beq
\label{A2}
\sup_{t \in [0,l]}\Arrowvert \nabla u_{\omega}(t)\Arrowvert_{L^2}\leqslant A(l)<1.
\eeq
Then, for all admissible pairs $(q,r)$ we have
\begin{eqnarray*}
\Arrowvert u_{\omega}-U\Arrowvert_{L^q((0,l),W^{1,r}(\R^2))}\underset{\arrowvert \omega\arrowvert\rightarrow \infty}{\longrightarrow}0,
\end{eqnarray*}
\end{lem}
The proof of Lemma \ref{main lem} is based on the Strichartz's estimate, the logarithmic and Moser-Trudinger inequalities and the fact that when $\arrowvert \omega\arrowvert$ approaches infinity, $\theta$ approaches its average.
This last observation is made more precisely as follows.
\begin{lem}
\label{a}
Let $(\gamma,\rho)$ be an admissible pairs and fix a time $t_0$. Given $f\in L^{\gamma'}(\R, L^{\rho'}(\R^2))$, we have
$$ \int_{t_0}^t \theta(\omega s) e^{i(t-s)\Delta}f(s)ds\underset{\arrowvert \omega\arrowvert\rightarrow \infty}{\longrightarrow} I(\theta)\int_{t_0}^te^{i(t-s)\Delta}f(s)ds \quad \mbox{in} \quad L^q(\R,L^r(\R^2)),
$$
for every admissible pairs $(q,r)$.
\end{lem}
\begin{proof}
See \cite{Cazenave}.
\end{proof}
The next lemma will also be used in the sequel.
\begin{lem}
\label{conv}
Set $f(u):= u(e^{4\pi\arrowvert u\arrowvert^2}-1).$ Then, for any $\varepsilon>0$, there exists a constant $C_\varepsilon>0$ such that
\begin{equation}
\label{Loc1}
|f(u)-f(v)|\leq C_\varepsilon |u-v|\bigg({\rm e}^{4\pi(1+\varepsilon)|u|^2}-1+ {\rm e}^{4\pi(1+\varepsilon)|v|^2}-1\bigg);
\end{equation}
and
\begin{eqnarray}
\label{Loc2}
|(Df)(u)-(Df)(v)|&\leq& C_\varepsilon |u-v|\bigg(|u|+{\rm e}^{4\pi(1+\varepsilon)|u|^2}-1+ |v|+{\rm e}^{4\pi(1+\varepsilon)|v|^2}-1\bigg).
\end{eqnarray}
\end{lem}
\begin{proof}
See \cite{JHDE}.
\end{proof}
For the proof of theorem \ref{EL}, the following refined estimates will be needed later on.
\begin{prop}
\label{prop2}
Suppose that $u_{\omega}$ satisfies \eqref{A2}, and let $[a,b]$ be a sub-interval of $[0,l]$. Then
$$
\|u_{\omega}(e^{4\pi\arrowvert u_{\omega}\arrowvert^2}-1)\|_{L^{\frac{4}{3}}((a,b), L^{\frac{4}{3}})}
\leqslant C(l) \|u_{\omega}\|_{L^4((a,b),W^{1,4})} \left(\|u_{\omega}\|_{_{L^4((a,b),W^{1,4})}}^{2}+\|u_{\omega}\|_{_{L^4((a,b),W^{1,4})}}^{4}\right)^{\alpha}
$$
and
$$
\|\nabla \left(u_{\omega}(e^{4\pi\arrowvert u_{\omega}\arrowvert^2}-1) \right)\|_{L^{\frac{4}{3}}((a,b), L^{\frac{4}{3}})}
\leqslant C(l) \|u_{\omega}\|_{L^4((a,b),W^{1,4})} \left(\|u_{\omega}\|_{_{L^4((a,b),W^{1,4})}}^{2}+\|u_{\omega}\|_{_{L^4((a,b),W^{1,4})}}^{4}\right)^{\beta}
$$
where $\alpha, \beta >0$ depend on $A(l)$.
\end{prop}
\begin{rem}
We note that, from the Strichartz's estimate, if $u_{\omega}$ exists on $(a,b)$ then it belongs to the space $L^4((a,b),W^{1,4})$.
\end{rem}
\begin{proof}

We begin by estimating $\|u_{\omega}(e^{4\pi\arrowvert u_{\omega}\arrowvert^2}-1)\|_{L^{\frac{4}{3}}((a,b), L^{\frac{4}{3}})}$.
Using H\"older inequality in space and time we get
$$
\|u_{\omega}(e^{4\pi\arrowvert u_{\omega}\arrowvert^2}-1)\|_{L^{\frac{4}{3}}((a,b), L^{\frac{4}{3}})}
\leq \|u_{\omega}\|_{L^4((a,b),L^4)} \|e^{4\pi\| u_{\omega}(t,\cdot)\|_{L^{\infty}_x}^2}-1\|_{L^{\frac{2}{\gamma}}(a,b)}^{\frac{1}{2}}  \|e^{4\pi\arrowvert u_{\omega}\arrowvert^2}-1\|_{L^{\frac{2}{2-\gamma}}((a,b),L^1)}^{\frac{1}{2}}
$$
where $0<\gamma<2$ is to be chosen suitably.\\
The assumption on $u_{\omega}$, Moser-Trudinger inequality and the conservation of mass give
$$
\|e^{4\pi\arrowvert u_{\omega}\arrowvert^2}-1\|_{L^{\frac{2}{2-\gamma}}((a,b),L^1)}^{\frac{1}{2}} \leqslant C l^{\frac{2-\gamma
}{4}} \|u_0\|_{L^2}.
$$
Now, write
\begin{eqnarray*}
\|e^{4\pi\| u_{\omega}(t,\cdot)\|_{L^{\infty}_x}^2}-1\|_{L^{\frac{2}{\gamma}}(a,b)}^{\frac{2}{\gamma}}&=& \int_{ \{t\in [a,b]\, ; \, \| u_{\omega}(t,\cdot)\|_{L^{\infty}_x} \leqslant 1\}}\left(e^{4\pi\| u_{\omega}(t,\cdot)\|_{L^{\infty}_x}^2}-1\right)^{\frac{2}{\gamma}} dt\\
&+&\int_{ \{t\in [a,b]\, ; \, \| u_{\omega}(t,\cdot)\|_{L^{\infty}_x} > 1\}} \left(e^{4\pi\| u_{\omega}(t,\cdot)\|_{L^{\infty}_x}^2}-1\right)^{\frac{2}{\gamma}} dt\\
&\leqslant& \int_{ \{t\in [a,b]\, ; \, \| u_{\omega}(t,\cdot)\|_{L^{\infty}_x} \leqslant 1\}}\left(e^{4\pi\| u_{\omega}(t,\cdot)\|_{L^{\infty}_x}^2}-1\right)^{\frac{2}{\gamma}} dt\\
&+& \int_{ \{t\in [a,b]\, ; \, \| u_{\omega}(t,\cdot)\|_{L^{\infty}_x} > 1\}} \left(e^{2\pi\| u_{\omega}(t,\cdot)\|_{L^{\infty}_x}^2}\right)^{\frac{4}{\gamma}} dt
\end{eqnarray*}
It can easily be shown that
$$
\int_{ \{t\in [a,b]\, ; \, \| u_{\omega}(t,\cdot)\|_{L^{\infty}_x} \leqslant 1\}}\left(e^{4\pi\| u_{\omega}(t,\cdot)\|_{L^{\infty}_x}^2}-1\right)^{\frac{2}{\gamma}} dt \leqslant C(\gamma) l^{\frac{1}{2}} \|u_{\omega}\|_{_{L^4((a,b),C^{\frac{1}{2}})}}^{2}.
$$

Indeed, let $t\in (a,b)$ be such that $\| u_{\omega}(t,\cdot)\|_{L^{\infty}_x} \leqslant 1$.  We have
\[e^{4\pi\| u_{\omega}(t,\cdot)\|_{L^{\infty}_x}^2}-1=\psi(\| u_{\omega}(t,\cdot)\|_{L^{\infty}_x})-\psi(0) \leq \bigg \lbrace\underset{s \in [0,\| u_{\omega}(t,\cdot)\|_{L^{\infty}_x}]}{\textit{sup}}|\psi'(s)|\bigg \rbrace \| u_{\omega}(t,\cdot)\|_{L^{\infty}_x}^2,\]
where $\psi(s):=e^{4\pi s }$. Note that, for all $s \in [0,\| u_{\omega}(t,\cdot)\|_{L^{\infty}_x}]$, $0\leq \psi'(s)\leq 4 \pi e^{4 \pi}:=C$. Therefore
\[\int_{ \{t\in [a,b]\, ; \, \| u_{\omega}(t,\cdot)\|_{L^{\infty}_x} \leqslant 1\}}\left(e^{4\pi\| u_{\omega}(t,\cdot)\|_{L^{\infty}_x}^2}-1\right)^{\frac{2}{\gamma}} dt \leq C^{\frac{2}{\gamma}} \int_{ \{t\in [a,b]\, ; \, \| u_{\omega}(t,\cdot)\|_{L^{\infty}_x} \leqslant 1\}} \| u_{\omega}(t,\cdot)\|_{L^{\infty}_x}^{\frac{2}{\gamma}} dt.\]

Since $\frac{4}{\gamma} \geq 2$ and $C^{\frac{1}{2}} \hookrightarrow L^{\infty}$, we get
\begin{align*}
\int_{ \{t\in [a,b]\, ; \, \| u_{\omega}(t,\cdot)\|_{L^{\infty}_x} \leqslant 1\}}\left(e^{4\pi\| u_{\omega}(t,\cdot)\|_{L^{\infty}_x}^2}-1\right)^{\frac{2}{\gamma}} dt &\leq& C^{\frac{2}{\gamma}}\int_{ \{t\in [a,b]\, ; \, \| u_{\omega}(t,\cdot)\|_{L^{\infty}_x} \leqslant 1\}}\|u_{\omega}(t,\cdot)\|_{C^{\frac{1}{2}}}^{2} dt\\
&\leq& C^{\frac{2}{\gamma}} \int_0^l \|u_{\omega}(t,\cdot)\|_{C^{\frac{1}{2}}}^{2} dt.
\end{align*}
We conclude using the Cauchy-Schwarz inequality.\\

Let $\epsilon>0$ (to be chosen later). We have
$$
\int_{ \{t\in [a,b]\, ; \, \| u_{\omega}(t,\cdot)\|_{L^{\infty}_x} > 1\}} \left(e^{2\pi\| u_{\omega}(t,\cdot)\|_{L^{\infty}_x}^2}\right)^{\frac{4}{\gamma}} dt\leqslant  \int_{ \{t\in [a,b]\, ; \, \| u_{\omega}(t,\cdot)\|_{L^{\infty}_x} > 1\}} \left(e^{2\pi(1+\epsilon)\| u_{\omega}(t,\cdot)\|_{L^{\infty}_x}^2}\right)^{\frac{4}{\gamma}}
$$

The log estimate and the assumption on $u_{\omega}$  allow us to find a constant $0<\gamma<2$ as desired such that
$$
\int_{ \{t\in [a,b]\, ; \, \| u_{\omega}(t,\cdot)\|_{L^{\infty}_x} > 1\}} \left(e^{2\pi\| u_{\omega}(t,\cdot)\|_{L^{\infty}_x}^2}\right)^{\frac{4}{\gamma}} dt\leqslant C(l,\gamma) \|u_{\omega}\|_{_{L^4((a,b),C^{\frac{1}{2}})}}^{4}.
$$
Indeed, let $t\in [a,b]$ be such that $\| u_{\omega}(t,\cdot)\|_{L^{\infty}_x} > 1$. Write the Log-estimate with $\beta=\frac{1}{2}$, $\lambda>\frac{1}{\pi}$ and $\mu \in (0,1]$ ( the latter two parameters are to be chosen later)
\begin{equation*}
\|u_{\omega}(t,\cdot)\|^2_{L^\infty}\leq
\lambda\|u_{\omega}(t,\cdot)\|_\mu^2 \log\left(C_{\lambda} +\left(\frac{8}{\mu}\right)^{\frac{1}{2}}
\frac{\|u_{\omega}(t,\cdot)\|_{{\mathcal C}^{\frac{1}{2}}}}{\|u_{\omega}(t,\cdot)\|_\mu}\right).
\end{equation*}
Since $A(l)<1$, one can choose $\mu \in (0,1]$ (independently of $t$) such that $A'(l,\mu)^2:=A(l)^2+\mu^2 M^2(u_0)<1$. Therefore
\begin{equation}
\label{previous}
\|u_{\omega}(t,\cdot)\|_\mu \leq A'(l,\mu).
\end{equation}
Now, it remains to choose $\lambda$ suitably. Note that for fixed $t$ and $\lambda$, the function $x \mapsto x^2 \log\left(C_{\lambda} +\left(\frac{8}{\mu}\right)^{\frac{1}{2}}
\frac{\|u_{\omega}(t,\cdot)\|_{{\mathcal C}^{\frac{1}{2}}}}{x}\right)$ defined for $x>0$ is increasing, hence from \eqref{previous} one comes to
\begin{equation*}
\|u_{\omega}(t,\cdot)\|^2_{L^\infty}\leq \lambda
A'(l,\mu)^2 \log\left(C_{\lambda} +\left(\frac{8}{\mu}\right)^{\frac{1}{2}}
\frac{\|u_{\omega}(t,\cdot)\|_{{\mathcal C}^{\frac{1}{2}}}}{A'(l,\mu)}\right),
\end{equation*}
and then
\begin{align}
\label{rew}
\nonumber
e^{2\pi(1+\epsilon)\| u_{\omega}(t,\cdot)\|_{L^{\infty}_x}^2} &\leq& \left(C_{\lambda} +\left(\frac{8}{\mu}\right)^{\frac{1}{2}}
\frac{\|u_{\omega}(t,\cdot)\|_{{\mathcal C}^{\frac{1}{2}}}}{A'(l,\mu)}\right)^{2\pi(1+\epsilon)\lambda
A'(l,\mu)^2} \\
&\leq& C(\lambda, \mu, l)^{2\pi(1+\epsilon)\lambda
A'(l,\mu)^2} \left(1+\|u_{\omega}(t,\cdot)\|_{{\mathcal C}^{\frac{1}{2}}}\right)^{2\pi(1+\epsilon)\lambda
A'(l,\mu)^2}.
\end{align}
Since $A'(l,\mu)<1$ one can choose $\epsilon>0$ such that $1+\epsilon<\frac{1}{A'(l,\mu)^2}$ and $\lambda>\frac{1}{\pi}$ such that $\lambda< \frac{1}{(1+\epsilon)A'(l,\mu)}$. With all parameters fixe, we set $\gamma:=2\pi(1+\epsilon)\lambda
A'(l,\mu)^2$. Note that $0<\gamma<2$ as claimed. The estimate \eqref{rew} can be rewritten as follows
\[e^{2\pi(1+\epsilon)\| u_{\omega}(t,\cdot)\|_{L^{\infty}_x}^2} \leq C(l)\left(1+\|u_{\omega}(t,\cdot)\|_{{\mathcal C}^{\frac{1}{2}}}\right)^{\gamma}.\]
Integrating the above inequality yields
\[\int_{ \{t\in [a,b]\, ; \, \| u_{\omega}(t,\cdot)\|_{L^{\infty}_x} > 1\}} \left(e^{2\pi\| u_{\omega}(t,\cdot)\|_{L^{\infty}_x}^2}\right)^{\frac{4}{\gamma}} dt\leqslant C(l,\gamma) \int_{ \{t\in [a,b]\, ; \, \| u_{\omega}(t,\cdot)\|_{L^{\infty}_x} > 1\}} \left(1+\|u_{\omega}(t,\cdot)\|_{{\mathcal C}^{\frac{1}{2}}}\right)^4 dt.\]
We conclude using the fact that $\|u_{\omega}(t,\cdot)\|_{{\mathcal C}^{\frac{1}{2}}}\geq \| u_{\omega}(t,\cdot)\|_{L^{\infty}_x} > 1$.

At final, we get
$$
\|e^{4\pi\| u_{\omega}(t,\cdot)\|_{L^{\infty}_x}^2}-1\|_{L^{\frac{2}{\gamma}}(a,b)} \leqslant C(l) \left(\|u_{\omega}\|_{_{L^4((a,b),C^{\frac{1}{2}})}}^{2}+\|u_{\omega}\|_{_{L^4((a,b),C^{\frac{1}{2}})}}^{4}\right)^{\frac{\gamma}{2}}
$$
We note that when $\|u_{\omega}\|_{_{L^4((a,b),C^{\frac{1}{2}})}}\leqslant1$, the above estimate reduces to
$$
\|e^{4\pi\| u_{\omega}(t,\cdot)\|_{L^{\infty}_x}^2}-1\|_{L^{\frac{2}{\gamma}}(a,b)} \leqslant C(l) \|u_{\omega}\|_{_{L^4((a,b),C^{\frac{1}{2}})}}^{\gamma}.
$$
Therefore,
$$
\|u_{\omega}(e^{4\pi\arrowvert u_{\omega}\arrowvert^2}-1)\|_{L^{\frac{4}{3}}((a,b), L^{\frac{4}{3}})}
\leqslant C(l) \|u_{\omega}\|_{L^4((a,b),L^4)} \left(\|u_{\omega}\|_{_{L^4((a,b),C^{\frac{1}{2}})}}^{2}+\|u_{\omega}\|_{_{L^4((a,b),C^{\frac{1}{2}})}}^{4}\right)^{\frac{\gamma}{4}}
$$

The Sobolev injection $W^{1,4}(\mathbb{R}^2)\hookrightarrow C^{\frac{1}{2}}(\mathbb{R}^2)$ concludes the proof of the first estimate.\\

Let us establish an analogous estimate for $\|\nabla \left( u_{\omega}(e^{4\pi\arrowvert u_{\omega}\arrowvert^2}-1)\right)\|_{L^{\frac{4}{3}}((a,b), L^{\frac{4}{3}})}$.\\
Before doing so, a straightforward calculation give
$$
|\nabla \left( u_{\omega}(e^{4\pi\arrowvert u_{\omega}\arrowvert^2}-1)\right)|\leqslant C |\nabla u_{\omega}| \left(e^{4\pi\arrowvert u_{\omega}\arrowvert^2}-1+|u_{\omega}|^{2}e^{4\pi\arrowvert u_{\omega}\arrowvert^2} \right).
$$

H\"older inequality, the above identity and the conservation of mass for $u_{\omega}$ give
\begin{eqnarray*}
\|\nabla \left( u_{\omega}(e^{4\pi\arrowvert u_{\omega}\arrowvert^2}-1)\right)\|_{L^{\frac{4}{3}}((a,b), L^{\frac{4}{3}})} &\lesssim& \|\nabla u_{\omega}\|_{L^4((a,b),L^4)} \|e^{4\pi\arrowvert u_{\omega}\arrowvert^2}-1\|_{L^2((a,b),L^2)}\\
& & +\||\nabla u_{\omega}||u_{\omega}|^{2}e^{4\pi\arrowvert u_{\omega}\arrowvert^2} \|_{L^{\frac{4}{3}}((a,b), L^{\frac{4}{3}})}\\
\end{eqnarray*}
We will only deal with the second term, the other one was treated above.\\

Recall that for any $\epsilon>0$ and $x\geqslant0$
$$
xe^x\leqslant \frac{e^{(1+\epsilon)x}-1}{\epsilon}.
$$
So
\begin{eqnarray*}
\||\nabla u_{\omega}||u_{\omega}|^{2}e^{4\pi\arrowvert u_{\omega}\arrowvert^2} \|_{L^{\frac{4}{3}}_{x}} &\leqslant& C(\epsilon) \||\nabla u_{\omega}|\left(e^{4\pi(1+\epsilon)\arrowvert u_{\omega}\arrowvert^2}-1\right) \|_{L^{\frac{4}{3}}_{x}}\\
&\leqslant& C(\epsilon) \|\nabla u_{\omega}\|_{L^4_x} \|e^{4\pi(1+\epsilon)\arrowvert u_{\omega}\arrowvert^2}-1\|_{L^2_x}\\
&\leqslant& C(\epsilon) \|\nabla u_{\omega}\|_{L^4_x} \left(e^{4\pi(1+\epsilon)\|u_{\omega}\|_{L^{\infty}_x}^2}-1\right)^{\frac{1}{2}} \|e^{4\pi(1+\epsilon)\arrowvert u_{\omega}\arrowvert^2}-1\|_{L^1_x}^{\frac{1}{2}}\\
&\leqslant& C(\epsilon) \|\nabla u_{\omega}\|_{L^4_x} \left(e^{4\pi(1+\epsilon)\|u_{\omega}\|_{L^{\infty}_x}^2}-1\right)^{\frac{1}{2}} C(\epsilon,A) \|u_0\|_{L^2}
\end{eqnarray*}
where in the last line we used Moser-Trudinger inequality for $\epsilon>0$ such that $\epsilon<\frac{1}{A^2}-1$ (a priori condition on $\epsilon$). Therefore
\begin{eqnarray*}
\||\nabla u_{\omega}||u_{\omega}|^{2}e^{4\pi\arrowvert u_{\omega}\arrowvert^2} \|_{L^{\frac{4}{3}}_{t,x}} &\leqslant& C(\epsilon,A) \|\nabla u_{\omega}\|_{L^4((a,b),L^4)} \|e^{4\pi(1+\epsilon)\|u_{\omega}\|_{L^{\infty}_x}^2}-1\|_{L^1(a,b)}^{\frac{1}{2}}.
\end{eqnarray*}
Let $0<\delta<2$ (to be chosen later). H\"older inequality in time gives
\begin{eqnarray*}
\||\nabla u_{\omega}||u_{\omega}|^{2}e^{4\pi\arrowvert u_{\omega}\arrowvert^2} \|_{L^{\frac{4}{3}}_{x}} &\leqslant& C(\epsilon,A) l^{\frac{2-\delta}{4}}\|\nabla u_{\omega}\|_{L^4((a,b),L^4)} \|e^{4\pi(1+\epsilon)\|u_{\omega}\|_{L^{\infty}_x}^2}-1\|_{L^{\frac{2}{\delta}}(a,b)}^{\frac{1}{2}}
\end{eqnarray*}
Now, write
\begin{eqnarray*}
\|e^{4(1+\epsilon)\pi\| u_{\omega}(t,\cdot)\|_{L^{\infty}_x}^2}-1\|_{L^{\frac{2}{\delta}}(a,b)}^{\frac{2}{\delta}}&=& \int_{ \{t\in [a,b]\, ; \, \| u_{\omega}(t,\cdot)\|_{L^{\infty}_x} \leqslant 1\}}\left(e^{4\pi(1+\epsilon)\| u_{\omega}(t,\cdot)\|_{L^{\infty}_x}^2}-1\right)^{\frac{2}{\delta}} dt\\
&+&\int_{ \{t\in [a,b]\, ; \, \| u_{\omega}(t,\cdot)\|_{L^{\infty}_x} > 1\}} \left(e^{4\pi(1+\epsilon)\| u_{\omega}(t,\cdot)\|_{L^{\infty}_x}^2}-1\right)^{\frac{2}{\delta}} dt\\
&\leqslant& \int_{ \{t\in [a,b]\, ; \, \| u_{\omega}(t,\cdot)\|_{L^{\infty}_x} \leqslant 1\}}\left(e^{4\pi(1+\epsilon)\| u_{\omega}(t,\cdot)\|_{L^{\infty}_x}^2}-1\right)^{\frac{2}{\delta}} dt\\
&+& \int_{ \{t\in [a,b]\, ; \, \| u_{\omega}(t,\cdot)\|_{L^{\infty}_x} > 1\}} \left(e^{2\pi(1+\epsilon)\| u_{\omega}(t,\cdot)\|_{L^{\infty}_x}^2}\right)^{\frac{4}{\delta}} dt.
\end{eqnarray*}
Arguing as previously, one gets
$$
\|\nabla \left( u_{\omega}(e^{4\pi\arrowvert u_{\omega}\arrowvert^2}-1)\right)\|_{L^{\frac{4}{3}}((a,b), L^{\frac{4}{3}})} \leqslant C(l) \|\nabla u_{\omega}\|_{L^4((a,b),L^4)} \left(\|u_{\omega}\|_{_{L^4((a,b),C^{\frac{1}{2}})}}^{2}+\|u_{\omega}\|_{_{L^4((a,b),C^{\frac{1}{2}})}}^{4}\right)^{\frac{\delta}{4}}.
$$

\end{proof}

Using the same technique as in the proof of Proposition \ref{prop2} we establish the following estimates.
\begin{prop}
\label{prop1}
Under the same hypothesis of lemma \ref{main lem}, let $[a,b]$ be a sub-interval of $[0,l]$. Then
\begin{equation}
\|e^{4\pi(1+\epsilon)\arrowvert u_{\omega}\arrowvert^2}-1\|_{L^2((a,b),L^2)} \leqslant C(l) \left(\|u_{\omega}\|_{_{L^4((a,b),C^{\frac{1}{2}})}}^{2}+\|u_{\omega}\|_{_{L^4((a,b),C^{\frac{1}{2}})}}^{4}\right)^{\alpha};
\end{equation}
\begin{equation}
\|e^{4\pi\arrowvert u_{\omega}\arrowvert^2}-1\|_{L^2((a,b),L^2)} \leqslant C(l) \left(\|u_{\omega}\|_{_{L^4((a,b),C^{\frac{1}{2}})}}^{2}+\|u_{\omega}\|_{_{L^4((a,b),C^{\frac{1}{2}})}}^{4}\right)^{\beta};
\end{equation}
\begin{equation}
\|\arrowvert u_{\omega}\arrowvert^2 e^{4\pi\arrowvert u_{\omega}\arrowvert^2}\|_{L^2((a,b),L^2)} \leqslant C(l)  \left(\|u_{\omega}\|_{_{L^4((a,b),C^{\frac{1}{2}})}}^{2}+\|u_{\omega}\|_{_{L^4((a,b),C^{\frac{1}{2}})}}^{4}\right)^{\gamma};
\end{equation}
\begin{equation}
\|e^{4\pi(1+\epsilon)\arrowvert u_{\omega}\arrowvert^2}-1\|_{L^{\frac{4}{3}}((a,b),L^{4(1+\epsilon)})} \leqslant C(l)  \left(\|u_{\omega}\|_{_{L^4((a,b),C^{\frac{1}{2}})}}^{2}+\|u_{\omega}\|_{_{L^4((a,b),C^{\frac{1}{2}})}}^{4}\right)^{\delta}.
\end{equation}
Here $\epsilon>0$ satisfies a finite number of smallness conditions and $\alpha, \beta, \gamma$ and $\delta$ are positive constants depending on $A(l)$ and $\epsilon$.
\end{prop}
\begin{rem}
The first and last estimates hold also true for $U$ under the hypothesis of Lemma \ref{main lem}.
\end{rem}
\subsection*{Proof of Lemma \ref{main lem}}
Define the function $f(u):= u(e^{4\pi\arrowvert u\arrowvert^2}-1).$ Divide the interval $[0,l]$ into a finite number of sub-intervals $[t_j,t_{j+1}],$ $j=0,...,J-1$, where $t_0=0$ and $t_J=l$. The integral forms for $u_{\omega}$ and $U$ read as follows
$$
u_{\omega}(t)=e^{i(t-t_j)\Delta}u_{\omega}(t_j)+i\int_{t_j}^t \theta (\omega s) e^{i(t-s)\Delta}f(u_{\omega}(s)) ds
$$
and
$$
U(t)=e^{i(t-t_j)\Delta}U(t_j)+i I(\theta)\int_{t_j}^t   e^{i(t-s)\Delta}f(U(s)) ds.
$$
Our aim is to estimate $\|u_{\omega}-U\|_{L^q((t_j,t_{j+1}), W^{1,r})}$.
Using the above integral forms, write
$$
u_{\omega}-U=i(I_1+I_2)+e^{i(t-t_j)\Delta}\left(u_{\omega}(t_j)-U(t_j)\right)
$$
where
\begin{eqnarray*}I_1&:=&\int_{t_j}^t \theta(\omega s) e^{i(t-s)\Delta}(f(u_{\omega}(s))-f(U(s)))ds\\
\end{eqnarray*}
and
\begin{eqnarray*}I_2&:=&\int_{t_j}^t(\theta (\omega s)-I(\theta))e^{i(t-s)\Delta}f(U(s))ds\\
\end{eqnarray*}
Using the Strichartz's estimate we get
$$
\|u_{\omega}-U\|_{L^q((t_j,t_{j+1}), L^{r})} \lesssim \|u_{\omega}(t_j)-U(t_j)\|_{L^2_x}+\|f(u_{\omega})-f(U)\|_{L^{\frac{4}{3}}((t_j,t_{j+1}),L^{\frac{4}{3}})}+\epsilon_{\omega,j}(q,r)
$$
where
\begin{eqnarray}
\label{c}
 \epsilon_{\omega,j}(q,r):=\Arrowvert I_2\Arrowvert_{L^q((t_j,t_{j+1}),L^r)}.
\end{eqnarray}
From Lemma \ref{a}, we infer
$$
\epsilon_{\omega,j}(q,r) \underset{\arrowvert \omega\arrowvert\rightarrow \infty}{\longrightarrow}0\quad\mbox{ for all}\quad j.
$$
To estimate the term $\|f(u_{\omega})-f(U)\|_{L^{\frac{4}{3}}((t_j,t_{j+1}),L^{\frac{4}{3}})}$, we use \eqref{Loc1}
for $\epsilon>0$ (to be chosen later suitably)
\begin{eqnarray*}
\|f(u_{\omega})-f(U)\|_{L^{\frac{4}{3}}((t_j,t_{j+1}),L^{\frac{4}{3}})} &\leqslant& C_{\varepsilon} \|u_{\omega}-U\|_{L^{4}((t_j,t_{j+1}),L^{4})} X_{\omega,j}\\
\end{eqnarray*}
where
$X_{\omega,j}:= \|e^{4\pi(1+\varepsilon)|u_{\omega}|^2}-1\|_{L^{2}((t_j,t_{j+1}),L^{2})}+\|e^{4\pi(1+\varepsilon)|U|^2}-1\|_{L^{2}((t_j,t_{j+1}),L^{2})}$.\\
At final we come to
$$
\|u_{\omega}-U\|_{L^q((t_j,t_{j+1}), L^{r})} \lesssim \|u_{\omega}(t_j)-U(t_j)\|_{L^2_x}+C_{\varepsilon} \|u_{\omega}-U\|_{L^{4}((t_j,t_{j+1}),L^{4})} X_{\omega,j}+\epsilon_{\omega,j}(q,r).
$$
We do the same for $\|\nabla \left(u_{\omega}-U\right)\|_{L^q((t_j,t_{j+1}), L^{r})}$.\\
A straightforward calculation give
$$
\nabla [f(u)]=(Df)(u) \cdot Du,
$$
where
$$
(Df)(u):=\left( \begin{array}{c}
e^{4\pi|u|^2}-1+ 4 \pi  |u|^2 e^{4\pi|u|^2}\\
4 \pi |u|^2 e^{4\pi|u|^2}\\
\end{array} \right)
$$
and
$$
Du:=\left( \begin{array}{c}
\nabla u\\
 \nabla \bar{u}\\
\end{array} \right).
$$
Using integral forms we get
$$
\nabla u_{\omega}-\nabla U=i(J_1+J_2+J_3)+e^{i(t-t_j)\Delta}\left( \nabla u_{\omega}(t_j)-\nabla U(t_j)\right)
$$
where
$$
J_1:=\int_{t_j}^t \theta(\omega s)\,  e^{i(t-s)\Delta} (Df)(u_{\omega}) \cdot (D u_{\omega}-D U)ds,
$$

$$
J_2:=\int_{t_j}^t \theta(\omega s)\, e^{i(t-s)\Delta}[(Df)(u_{\omega})- (Df)(U)] \cdot D U ds,
$$
and
$$
J_3:=\int_{t_j}^t[\theta (\omega s)-I(\theta)] \, e^{i(t-s)\Delta}  \nabla [f(U)] ds.
$$
Using Strichartz's estimate we get
\begin{eqnarray*}
\|\nabla \left(u_{\omega}-U\right)\|_{L^q((t_j,t_{j+1}), L^{r})} &\lesssim& \|\nabla \left(u_{\omega}(t_j)-U(t_j)\right)\|_{L^2_x}+\|(Df)(u_{\omega}) \cdot (D u_{\omega}-D U)\|_{L^{\frac{4}{3}}((t_j,t_{j+1}),L^{\frac{4}{3}})}\\
&+&\|[(Df)(u_{\omega})- (Df)(U)] \cdot D U\|_{L^1((t_j,t_{j+1}), L^{2})} +\tilde{\epsilon}_{\omega,j}(q,r)
\end{eqnarray*}
where
\begin{eqnarray}
\label{d}
 \tilde{\epsilon}_{\omega,j}(q,r):=\Arrowvert J_3\Arrowvert_{L^q((t_j,t_{j+1}),L^r)}.
\end{eqnarray}
From Lemma \ref{a}, we infer
$$
\tilde{\epsilon}_{\omega,j}(q,r) \underset{\arrowvert \omega\arrowvert\rightarrow \infty}{\longrightarrow}0\quad\mbox{ for all}\quad j.
$$
On one hand, we have
$$
\|(Df)(u_{\omega}) \cdot (D u_{\omega}-D U)\|_{L^{\frac{4}{3}}((t_j,t_{j+1}),L^{\frac{4}{3}})} \lesssim \|\nabla u_{\omega}- \nabla U \|_{L^4((t_j,t_{j+1}),L^4)} Y_{\omega,j}.
$$
Here $Y_{\omega,j}:= \|e^{4\pi|u_{\omega}|^2}-1\|_{L^{2}((t_j,t_{j+1}),L^{2})}+\||u_{\omega}|^2e^{4\pi|u_{\omega}|^2}\|_{L^{2}((t_j,t_{j+1}),L^{2})}$.\\
On the other hand, estimate \eqref{Loc2} yields
$$
\|[(Df)(u_{\omega})- (Df)(U)] \cdot D U\|_{L^1((t_j,t_{j+1}), L^{2})} \lesssim C_{\varepsilon} \|u_{\omega}-U\|_{L^{\infty}((t_j,t_{j+1}),H^{1})} Z_{\omega,j}.
$$
where
\begin{eqnarray*}
Z_{\omega,j}&:=&\left(\|u_{\omega}\|_{L^{4}((t_j,t_{j+1}),H^{1})}+\|U\|_{L^{4}((t_j,t_{j+1}),H^{1})}+\|e^{4\pi(1+\epsilon)\arrowvert u_{\omega}\arrowvert^2}-1\|_{L^{\frac{4}{3}}((t_j,t_{j+1}),L^{4(1+\epsilon)}_x)}\right.\\
& &+\left.\|e^{4\pi(1+\epsilon)\arrowvert U\arrowvert^2}-1\|_{L^{\frac{4}{3}}((t_j,t_{j+1}),L^{4(1+\epsilon)}_x)}\right) \|\nabla U\|_{L^{4}((t_j,t_{j+1}),L^{4})},
\end{eqnarray*}
and $\epsilon>0$ to be chosen suitably. Here we used the Sobolev injection $H^1(\mathbb{R}^2)\hookrightarrow L^8(\mathbb{R}^2)$ and the embedding $L^{4}((t_j,t_{j+1}))\hookrightarrow L^{\frac{4}{3}}((t_j,t_{j+1}))$. Moreover
\begin{eqnarray*}
\|\nabla \left(u_{\omega}-U\right)\|_{L^q((t_j,t_{j+1}), L^{r})} &\lesssim& \|\nabla \left(u_{\omega}(t_j)-U(t_j)\right)\|_{L^2_x}+\|\nabla u_{\omega}- \nabla U \|_{L^4((t_j,t_{j+1}),L^4)} Y_{\omega,j}\\
&+& C_{\varepsilon} \|u_{\omega}-U\|_{L^{\infty}((t_j,t_{j+1}),H^{1})} Z_{\omega,j} +\tilde{\epsilon}_{\omega,j}(q,r).
\end{eqnarray*}
Summing the inequalities we get
\begin{eqnarray*}
\|u_{\omega}-U\|_{L^q((t_j,t_{j+1}), W^{1,r})} &\lesssim&\|u_{\omega}(t_j)-U(t_j)\|_{H^1_x}+C_{\varepsilon} \|u_{\omega}-U\|_{L^{4}((t_j,t_{j+1}),W^{1,4})}Y_{\omega,j}\\
&+&\|u_{\omega}-U\|_{L^{\infty}((t_j,t_{j+1}),H^{1})} (X_{\omega,j}+Z_{\omega,j})+\left(\epsilon_{\omega,j}(q,r)+\tilde{\epsilon}_{\omega,j}(q,r)\right).
\end{eqnarray*}
Now we will use Proposition \ref{prop1} to estimate successively the quantities $X_{\omega,j},Y_{\omega,j}$ and $Z_{\omega,j}$.\\
Set
$$
X_{\omega}:= \|e^{4\pi(1+\varepsilon)|u_{\omega}|^2}-1\|_{L^{2}((0,l),L^{2})}+\|e^{4\pi(1+\varepsilon)|U|^2}-1\|_{L^{2}((0,l),L^{2})};
$$
$$
Y_{\omega}:= \|e^{4\pi|u_{\omega}|^2}-1\|_{L^{2}((0,l),L^{2})}+\||u_{\omega}|^2e^{4\pi|u_{\omega}|^2}\|_{L^{2}((0,l),L^{2})},
$$
\begin{eqnarray*}
Z_{\omega}&:=&\left(\|u_{\omega}\|_{L^{4}((0,l),H^{1})}+\|U\|_{L^{4}((0,l),H^{1})}+\|e^{4\pi(1+\epsilon)\arrowvert u_{\omega}\arrowvert^2}-1\|_{L^{\frac{4}{3}}((0,l),L^{4(1+\epsilon)}_x)}\right.\\
& &+\left.\|e^{4\pi(1+\epsilon)\arrowvert U\arrowvert^2}-1\|_{L^{\frac{4}{3}}((0,l),L^{4(1+\epsilon)}_x)}\right) \|\nabla U\|_{L^{4}((0,l),L^{4})};
\end{eqnarray*}
$$
\epsilon_{\omega}(q,r):=\Arrowvert I_2\Arrowvert_{L^q((0,l),L^r)} \quad \mbox{and}  \quad \tilde{\epsilon}_{\omega}(q,r):=\Arrowvert J_3\Arrowvert_{L^q((0,l),L^r)}.
$$
We have
$$
X_{\omega} \leqslant C(l) \bigg \lbrace \left(\|u_{\omega}\|_{_{L^4((0,l),C^{\frac{1}{2}})}}^{2}+\|u_{\omega}\|_{_{L^4((0,l),C^{\frac{1}{2}})}}^{4}\right)^{\alpha}+\left(\|U\|_{_{L^4((0,l),C^{\frac{1}{2}})}}^{2}+\|U\|_{_{L^4((0,l),C^{\frac{1}{2}})}}^{4}\right)^{\tilde{\alpha}} \bigg \rbrace,
$$
$$
Y_{\omega} \leqslant C(l) \bigg \lbrace \left(\|u_{\omega}\|_{_{L^4((0,l),C^{\frac{1}{2}})}}^{2}+\|u_{\omega}\|_{_{L^4((0,l),C^{\frac{1}{2}})}}^{4}\right)^{\beta}+\left(\|u_{\omega}\|_{_{L^4((0,l),C^{\frac{1}{2}})}}^{2}+\|u_{\omega}\|_{_{L^4((0,l),C^{\frac{1}{2}})}}^{4}\right)^{\gamma}\bigg \rbrace,
$$
and
\begin{eqnarray*}
Z_{\omega}&\leqslant& C(l) \bigg \lbrace\|u_{\omega}\|_{L^{4}((0,l),H^{1})}+\|U\|_{L^{4}((0,l),H^{1})}+ \left(\|u_{\omega}\|_{_{L^4((0,l),C^{\frac{1}{2}})}}^{2}+\|u_{\omega}\|_{_{L^4((0,l),C^{\frac{1}{2}})}}^{4}\right)^{\delta}\\
& &+\left( \|U\|_{_{L^4((0,l),C^{\frac{1}{2}})}}^{2}+\|U\|_{_{L^4((0,l),C^{\frac{1}{2}})}}^{4}\right)^{\tilde{\delta}}  \bigg \rbrace \|\nabla U\|_{L^{4}((0,l),L^{4})},
\end{eqnarray*}
where $\epsilon>0$ was chosen according to Proposition \ref{prop1}.

The hypothesis on $u_{\omega}$ and $U$ allow us to apply Lemma \ref{dec} and to divide the interval $[0,l]$ into a finite number of sub-intervals $[t_j,t_{j+1}],$ $j=0,...,J-1$, where $t_0=0$, $t_J=l$ and $J$ is a positive integer less than a constant independent of $\omega$ and such that for $|\omega|$ sufficiently large and all $j$
$$
X_{\omega,j}+Z_{\omega,j} \leqslant \frac{1}{2} \quad \mbox{and}  \quad Y_{\omega,j} \leqslant \frac{1}{6}.
$$
Let us give some details here. We will only consider the $Y_{\omega,j}$-estimate, the other one could be carried out similarly. \\
Let $\epsilon>0$ be such that $C(l)\lbrace \left(\epsilon+\epsilon^{\frac{1}{2}}\right)^{\beta}+\left(\epsilon+\epsilon^{\frac{1}{2}}\right)^{\gamma}\bigg \rbrace\leqslant \frac{1}{6}$.\\
Since $\underset{|\omega| \to \infty}{\mbox{lim sup}}\|u_{\omega}\|_{L^4(0,l), C^{\frac{1}{2}}(\R^2))}<\infty$, there exists $\xi_0$, such that for all $\xi\geq\xi_0$ and all $|\omega|\geq\xi$
\[\|u_{\omega}\|_{L^4(0,l), C^{\frac{1}{2}}(\R^2))} \leq \underset{|\omega| \to \infty}{\mbox{lim sup}}\|u_{\omega}\|_{L^4(0,l), C^{\frac{1}{2}}(\R^2))}+1.\]
Fix $\omega$ such that $|\omega|\geq\xi_0$ and set $h(t):=\|u_{\omega}(t,\cdot)\|_{ C^{\frac{1}{2}}(\R^2)}^4$ and $M:=\left(\underset{|\omega| \to \infty}{\mbox{lim sup}}\|u_{\omega}\|_{L^4(0,l), C^{\frac{1}{2}}(\R^2))}+1\right)^4$. The previous claim can be rewritten as follows
\[\int_0^l h(t) dt \leq M. \]
From Lemma \ref{dec}, there exists a finite partition of the interval $[0,l]$ into a family of sub-intervals $\{[t_j,t_{j+1}]\}_{j=0}^{J-1}$, where $t_0=0$, $t_J=l$, $J$ a positive integer less than $[\frac{M}{\epsilon}]+1$ and such that, for all $j\in \{0,\cdots,J-1\}$
\[\int_{t_j}^{t_{j+1}} h(t) dt \leq \epsilon. \]
We infer that, for all $j\in \{0,\cdots,J-1\}$
\begin{align*}
Y_{\omega,j} &\leqslant C(l) \bigg \lbrace \left(\|u_{\omega}\|_{_{L^4((t_j,t_{j+1}),C^{\frac{1}{2}})}}^{2}+\|u_{\omega}\|_{_{L^4((t_j,t_{j+1}),C^{\frac{1}{2}})}}^{4}\right)^{\beta}+\left(\|u_{\omega}\|_{_{L^4((t_j,t_{j+1}),C^{\frac{1}{2}})}}^{2}+\|u_{\omega}\|_{_{L^4((t_j,t_{j+1}),C^{\frac{1}{2}})}}^{4}\right)^{\gamma}\bigg \rbrace\\
&\leqslant C(l)\bigg \lbrace \left(\epsilon+\epsilon^{\frac{1}{2}}\right)^{\beta}+\left(\epsilon+\epsilon^{\frac{1}{2}}\right)^{\gamma}\bigg \rbrace\leqslant \frac{1}{6}.
\end{align*}
This achieves the proof of the claimed estimate on $Y_{\omega,j}$.\\
We note that, a priori, the integer $J$ as well as the real numbers $t_j$ may depend on $\omega$.\\
In the sequel we will denote $\epsilon_{\omega}(q,r)+\tilde{\epsilon}_{\omega}(q,r)$ by $\alpha_{\omega}(q,r)$.
We have, for all $j$
\begin{eqnarray*}
\|u_{\omega}-U\|_{L^q((t_j,t_{j+1}), W^{1,r})} &\lesssim&\|u_{\omega}(t_j)-U(t_j)\|_{H^1_x}+  \frac{1}{6} \|u_{\omega}-U\|_{L^{4}((t_j,t_{j+1}),W^{1,4})}\\
&+& \frac{1}{2} \|u_{\omega}-U\|_{L^{\infty}((t_j,t_{j+1}),H^{1})} +\alpha_{\omega}(q,r).
\end{eqnarray*}
We argue as follows. Letting $j=0$, yields
\begin{eqnarray*}
\|u_{\omega}-U\|_{L^q((t_0,t_{1}), W^{1,r})} &\lesssim&  \frac{1}{6} \|u_{\omega}-U\|_{L^{4}((t_0,t_{1}),W^{1,4})}+\frac{1}{2} \|u_{\omega}-U\|_{L^{\infty}((t_0,t_{1}),H^{1})} +\\
& & +\alpha_{\omega}(q,r).
\end{eqnarray*}
Letting $(q,r)=(\infty,2)$, we see that
\begin{eqnarray*}
\|u_{\omega}-U\|_{L^{\infty}((t_0,t_{1}),H^{1})} &\lesssim& 2 \bigg \lbrace \frac{1}{6} \|u_{\omega}-U\|_{L^{4}((t_0,t_{1}),W^{1,4})}+\alpha_{\omega}(\infty,2) \bigg \rbrace.
\end{eqnarray*}
Thus
\begin{eqnarray*}
\|u_{\omega}-U\|_{L^q((t_0,t_{1}), W^{1,r})} &\lesssim&  \frac{1}{3} \|u_{\omega}-U\|_{L^{4}((t_0,t_{1}),W^{1,4})}+\alpha_{\omega}(\infty,2)+\alpha_{\omega}(q,r).
\end{eqnarray*}
Letting $(q,r)=(4,4)$, we get
\begin{eqnarray*}
\|u_{\omega}-U\|_{L^4((t_0,t_{1}), W^{1,4})} &\lesssim&  \frac{3}{2}\left(\alpha_{\omega}(\infty,2)+\alpha_{\omega}(4,4)\right),
\end{eqnarray*}
and therefore,
\begin{eqnarray*}
\|u_{\omega}-U\|_{L^q((t_0,t_{1}), W^{1,r})} &\lesssim&  \frac{3}{2}\alpha_{\omega}(\infty,2) +\frac{1}{2}\alpha_{\omega}(4,4)+\alpha_{\omega}(q,r).
\end{eqnarray*}

An induction argument allows us to prove that, for all $j$ and all admissible pairs $(q,r)$
\begin{eqnarray}
 \label{Induction}
\|u_{\omega}-U\|_{L^q((t_j,t_{j+1}), W^{1,r})} &\lesssim&  a_j\alpha_{\omega}(\infty,2) +b_j\alpha_{\omega}(4,4)+\alpha_{\omega}(q,r).
\end{eqnarray}
where $a_j$ and $b_j$ are defined as follows
$$
a_j:=\frac{3^{j+1}}{2}+\frac{3^{j+2}}{4}-\frac{9}{4}, \quad j\in \lbrace {0,\cdots,J-1} \rbrace,
$$
and
$$
b_j:=\frac{3^{j}}{2}+\frac{3^{j}-1}{4}, \quad j\in \lbrace {0,\cdots,J-1} \rbrace.
$$
Indeed, if $J=1$, then the only value that could be taken by $j$ is $0$. This case was already settled above. Now, assume that $J\geq 2$ and let us prove the claimed estimate via an induction argument.\\
For $j=0$, there is nothing to prove. Assume that estimate \eqref{Induction} is true up to some $j<J-1$ and let us prove its validity for $j+1$. We have
\begin{eqnarray*}
\|u_{\omega}-U\|_{L^q((t_{j+1},t_{j+2}), W^{1,r})} &\lesssim&\|u_{\omega}(t_{j+1})-U(t_{j+1})\|_{H^1_x}+  \frac{1}{6} \|u_{\omega}-U\|_{L^{4}((t_{j+1},t_{j+2}),W^{1,4})}\\
&+& \frac{1}{2} \|u_{\omega}-U\|_{L^{\infty}((t_{j+1},t_{j+2}),H^{1})} +\alpha_{\omega}(q,r).
\end{eqnarray*}
Estimate \eqref{Induction} gives for $(q,r)=(\infty,2)$
\begin{eqnarray*}
\|u_{\omega}(t_{j+1})-U(t_{j+1})\|_{H^1_x} &\lesssim&  (a_j+1)\alpha_{\omega}(\infty,2) +b_j\alpha_{\omega}(4,4).
\end{eqnarray*}
Therefore
\begin{eqnarray*}
\|u_{\omega}-U\|_{L^q((t_{j+1},t_{j+2}), W^{1,r})} &\lesssim&(a_j+1)\alpha_{\omega}(\infty,2) +b_j\alpha_{\omega}(4,4)+  \frac{1}{6} \|u_{\omega}-U\|_{L^{4}((t_{j+1},t_{j+2}),W^{1,4})}\\
&+& \frac{1}{2} \|u_{\omega}-U\|_{L^{\infty}((t_{j+1},t_{j+2}),H^{1})} +\alpha_{\omega}(q,r).
\end{eqnarray*}
Letting $(q,r)=(\infty,2)$ in the latter estimate yields
\begin{eqnarray*}
\frac{1}{2} \|u_{\omega}-U\|_{L^{\infty}((t_{j+1},t_{j+2}),H^{1})} &\lesssim&(a_j+2)\alpha_{\omega}(\infty,2) +b_j\alpha_{\omega}(4,4)+  \frac{1}{6} \|u_{\omega}-U\|_{L^{4}((t_{j+1},t_{j+2}),W^{1,4})}.
\end{eqnarray*}
Hence
\begin{eqnarray*}
\|u_{\omega}-U\|_{L^q((t_{j+1},t_{j+2}), W^{1,r})} &\lesssim&(2 a_j+3)\alpha_{\omega}(\infty,2) + 2 b_j\alpha_{\omega}(4,4)+  \frac{1}{3} \|u_{\omega}-U\|_{L^{4}((t_{j+1},t_{j+2}),W^{1,4})}\\
&+& \alpha_{\omega}(q,r).
\end{eqnarray*}
Now let $(q,r)=(4,4)$ in the above inequality. One gets
\begin{eqnarray*}
\frac{1}{3} \|u_{\omega}-U\|_{L^{4}((t_{j+1},t_{j+2}),W^{1,4})}&\lesssim& \frac{1}{2 }\lbrace(2 a_j+3)\alpha_{\omega}(\infty,2) + (2 b_j+1)\alpha_{\omega}(4,4)\rbrace,
\end{eqnarray*}
so that
\begin{eqnarray*}
\|u_{\omega}-U\|_{L^q((t_{j+1},t_{j+2}), W^{1,r})} &\lesssim&(3 a_j+\frac{9}{2})\alpha_{\omega}(\infty,2) + (3 b_j+\frac{1}{2})\alpha_{\omega}(4,4)+ \alpha_{\omega}(q,r).
\end{eqnarray*}
We conclude noting that $a_{j+1}=3 a_j+\frac{9}{2}$ and $b_{j+1}=3 b_j+\frac{1}{2}$.

Since $J$ is less than a constant independent of $\omega$, we can bound $a_j$ and $b_j$ from above by a constant independent of $\omega$. Thus, for all $j$
\begin{eqnarray*}
\|u_{\omega}-U\|_{L^q((t_j,t_{j+1}), W^{1,r})} &\lesssim& \alpha_{\omega}(\infty,2) +\alpha_{\omega}(4,4)+\alpha_{\omega}(q,r).
\end{eqnarray*}
The fact that
$$
\lbrace\alpha_{\omega}(\infty,2) +\alpha_{\omega}(4,4)+\alpha_{\omega}(q,r) \rbrace \underset{\arrowvert \omega\arrowvert\rightarrow \infty}{\longrightarrow}0,
$$
implies (after summing over $j$ and bounding again $J$ independently of $\omega$ )
$$
\|u_{\omega}-U\|_{L^q(0,l), W^{1,r})}\underset{\arrowvert \omega\arrowvert\rightarrow \infty}{\longrightarrow}0.
$$
This achieves the proof of Lemma \ref{main lem}.

\section{Proof of the Main Result}

Now we are in position to prove Theorem \ref{EL}.
Fix a time $0<T<\infty$. Set
$N:=\|\theta\|_{L^{\infty}(\mathbb{R})}$. We can divide the interval $[0,T]$ into a finite number of sub-intervals $[t_{j},t_{j+1}]$, $j\in \{0\cdots J-1\}$ for some $J\geqslant1$ such that, for all $j$
$$
\|U\|_{L^4([t_{j},t_{j+1}], W^{1,4}(\mathbb{R}^2))}\leqslant \epsilon.
$$
Here $0<\epsilon<1$ is to be chosen and depending on $A(T)$, $T$ , $N$ and some constants from the Strichartz's estimates and H\"older inequality.\\
Using the integral form of $U$ on each time interval $[t_{j},t_{j+1}]$, the Strichartz's estimate and Proposition \ref{prop1} for $U$, we get
\begin{eqnarray*}
\|e^{i(\cdot-t_{j})\Delta}U(t_{j})\|_{L^4([t_{j},t_{j+1}], W^{1,4}(\mathbb{R}^2))}&\leqslant& \|U\|_{L^4([t_{j},t_{j+1}], W^{1,4}(\mathbb{R}^2))}\\
\end{eqnarray*}
$$
+C(T) N \| U\|_{L^4([t_{j},t_{j+1}], W^{1,4})} \bigg\lbrace\left(\| U\|_{L^4([t_{j},t_{j+1}], W^{1,4})}^{2}+\| U\|_{L^4([t_{j},t_{j+1}], W^{1,4})}^{4}\right)^{\mu}
$$
$$
+ \left(\| U\|_{L^4([t_{j},t_{j+1}], W^{1,4})}^{2}+\| U\|_{L^4([t_{j},t_{j+1}], W^{1,4})}^{4}\right)^{\nu}\bigg \rbrace,
$$
where $\mu,\nu>0$ depend on $H(u_0)$. We see that for $\epsilon>0$ small enough
$$
\|e^{i(\cdot-t_{j})\Delta}U(t_{j})\|_{L^4([t_{j},t_{j+1}], W^{1,4}(\mathbb{R}^2))} \leqslant 2 \epsilon.
$$
For $t \in [t_{0},t_{1}]$, we get using Strichartz's estimate
\begin{eqnarray}
\label{bs}
\|u_{\omega}\|_{L^4([t_{0},t], W^{1,4}(\mathbb{R}^2))}&\leqslant& \|e^{i \tau \Delta}u_{0}\|_{L^4([t_{0},t_{1}], W^{1,4})}\\
\nonumber
&+& C(T) N \| u_{\omega}\|_{L^4([t_{0},t], W^{1,4})} \bigg \lbrace\left(\| u_{\omega}\|_{L^4([t_{0},t], W^{1,4})}^{2}+\| u_{\omega}\|_{L^4([t_{0},t], W^{1,4})}^{4}\right)^{\alpha}\\
\nonumber
&+& \left(\| u_{\omega}\|_{L^4([t_{0},t], W^{1,4})}^{2}+\| u_{\omega}\|_{L^4([t_{0},t], W^{1,4})}^{4}\right)^{\beta}\bigg \rbrace.
\end{eqnarray}
Here $\alpha$ and $\beta$ depend on $A(T)$.
The continuity argument (see Appendix) allows us to conclude that, for all $t \in [t_{0},t_{1}]$
$$
\|u_{\omega}\|_{L^4([t_{0},t], W^{1,4}(\mathbb{R}^2))}\leqslant C(T,N, \alpha, \beta).
$$
Indeed, set $X(t):=\|u_{\omega}\|_{L^4([t_{0},t], W^{1,4}(\mathbb{R}^2))}$, $t \in [t_{0},t_{1}]$. One can check, using Lebesgue dominated convergence theorem, that the nonnegative function $X$ is continuous on $[t_{0},t_{1}]$ and satisfies
\[ X(t) \leqslant 2 \epsilon +C(T) N X(t) \lbrace (X(t)^2+X(t)^4)^{\alpha}+ (X(t)^2+X(t)^4)^{\beta}\rbrace . \]
We assume without loss of generality that $ \alpha \leq \beta$.
The function $x \mapsto 2 \epsilon +C(T) N x \lbrace (x^2+x^4)^{\alpha}+ (x^2+x^4)^{\beta}\rbrace $ has the same behavior as $x \mapsto 2 \epsilon +C(T, \alpha, \beta) x^{1+2\alpha}$ in a neighborhood of $0$ and as $x \mapsto C(T, \alpha, \beta) x^{1+4 \beta}$ in a neighborhood of $+\infty$
. Therefore, one could carry out the same proof as in Lemma \ref{cont} to infer that, for a suitable choice of $\epsilon$, we have
\[X(t) \leqslant C(T,N,\alpha, \beta),\]
for all $t \in [t_{0},t_{1}]$. Here $C(T,N,\alpha, \beta)$ is some constant depending on $T, N, \alpha$ and $\beta$. The Sobolev injection $W^{1,4}(\mathbb{R}^2)\hookrightarrow C^{\frac{1}{2}}(\mathbb{R}^2)$ gives
$$
\underset{|\omega| \to \infty}{\mbox{lim sup}} \|u_{\omega}\|_{L^4([t_{0},t_{1}], C^{\frac{1}{2}}(\mathbb{R}^2))}<\infty.
$$
 Hence, from the local theory, $u_{\omega}$ exists on $[t_{0},t_{1}]$ for $|\omega|$ sufficiently large. Lemma \ref{main lem} allows us to conclude in particular that
$$
\|u_{\omega}(t_1)-U(t_1)\|_{H^1} \underset{|\omega| \to \infty}{\rightarrow}0.
$$
On $[t_{1},t_{2}]$, we get arguing as above
\begin{eqnarray*}
\|u_{\omega}\|_{L^4([t_{1},t], W^{1,4}(\mathbb{R}^2))}&\leqslant&\|u_{\omega}(t_1)-U(t_1)\|_{H^1}+ \|e^{i(\cdot-t_{1})\Delta}U(t_{1})\|_{L^4([t_{1},t_{2}], W^{1,4})}\\
&+& C(T) N \| u_{\omega}\|_{L^4([t_{1},t], W^{1,4})} \bigg \lbrace\left(\| u_{\omega}\|_{L^4([t_{1},t], W^{1,4})}^{2}+\| u_{\omega}\|_{L^4([t_{1},t], W^{1,4})}^{4}\right)^{\alpha}\\
&+&\left(\| u_{\omega}\|_{L^4([t_{1},t], W^{1,4})}^{2}+\| u_{\omega}\|_{L^4([t_{1},t], W^{1,4})}^{4}\right)^{\beta}\bigg \rbrace.
\end{eqnarray*}
Again the continuity argument insures that
$$
\underset{|\omega| \to \infty}{\mbox{lim sup}}\|u_{\omega}\|_{L^4([t_{0},t_{2}], C^{\frac{1}{2}}(\mathbb{R}^2))}<\infty.
$$
Therefore, $u_{\omega}$ exists on $[t_{0},t_{2}]$ for $|\omega|$ sufficiently large and Lemma \ref{main lem} gives
$$
\|u_{\omega}(t_2)-U(t_2)\|_{H^1} \underset{|\omega| \to \infty}{\rightarrow}0.
$$
An induction argument achieves the proof of Theorem \ref{EL}.

\section{Appendix}
\begin{lem}
\label{dec}
Let $M, \,\ell>0$. Suppose that $f: [0,\ell] \rightarrow \R^{+}$ is an integrable and positive function satisfying
$$
\int_0^\ell f(t) \, dt \leqslant M.
$$
Then, for all $\epsilon>0$, there exists a finite partition of $[0,\ell]$ into a family of sub-intervals $\{[t_j,t_{j+1}]\}_{j=0}^{J-1}$, where $t_0=0$, $t_J=l$ and $J$ is a positive integer less than $[\frac{M}{\epsilon}]+1$ such that, for all $j\in \{0, \cdots, J-1\}$
$$
\int_{t_j}^{t_{j+1}}\, f(t) \, dt \leqslant \epsilon.
$$
Here $[x]$ denotes the integer part of the real number $x$.
\end{lem}
\begin{proof}
Set $\phi(x):=\displaystyle\int_0^x\, f(t) \, dt$, $0\leqslant x \leqslant \ell$. It is clear that $\phi$ is continuous and increasing. We distinguish two cases.\\
$(i)$ $M\leqslant \epsilon$:\\
In this case it suffices to take $J=1$, $t_0=0$ and $t_J=l$.\\
$(ii)$ $M> \epsilon$:\\
Set $N:=[\frac{M}{\epsilon}]$ the integer part of $\frac{M}{\epsilon}$.
\begin{itemize}
\item If $\phi(\ell) < N\epsilon$. Set $n:=[\frac{\phi(\ell)}{\epsilon}] \in \{0, \cdots, N-1\}$. We have

$$
\phi(\ell) \in [n \epsilon, (n+1) \epsilon[.
$$
The mean value theorem insures the following:\\
For all $j\in \{0, \cdots, n\}$, there exists $x_j \in [0,\ell]$ $(x_0=0)$ such that
$$
\phi(x_j)=j \epsilon
$$
It suffices now to take $t_0=0$, $t_1=x_1$, $\cdots$, $t_{J-1}=x_n$ and $t_J=\ell$.\\
We see that, in this case, $J=n+1 \leqslant N \leqslant [\frac{M}{\epsilon}]+1$.
\end{itemize}
\begin{itemize}
\item if $N\epsilon \leqslant \phi(\ell)$, we argue similarly.
\end{itemize}
\end{proof}

\begin{lem}[Continuity argument]
\label{cont}
Let $X : [0,T]\to\R$ be a nonnegative continuous, such that, for
every $0\leqslant t \leqslant T $,
$$
X(t) \leqslant a+ b X(t)^{\theta} \, ,
$$
where $a,b>0$ and $\theta>1$ are constants such that
$$
a<\left(1-\frac{1}{\theta}\right) \frac{1}{(\theta b)^{1/(\theta-1)}} \quad \mbox{and} \quad X(0)\leqslant \frac{1}{(\theta b)^{1/(\theta-1)}}.
$$
Then, for every $0\leqslant t \leqslant T $, we have
$$
X(t) \leqslant \frac{\theta}{\theta-1}a.
$$
\end{lem}
\begin{proof}
We sketch the proof for the convenience of the reader.\newline The function $f: x \longmapsto b x^{\theta}-x+a$ is decreasing on $[0,(\theta b)^{1/(1-\theta)} ]$ and increasing on $[(\theta b)^{1/(1-\theta)},\infty[$. The assumptions on $a$ and $X(0)$ imply that $f((\theta b)^{1/(1-\theta)})<0$. As $f(X(t))\geqslant 0, f(0) >0$ and $X(0)\leqslant \frac{1}{(\theta b)^{1/(\theta-1)}}$, we deduce the desired result.
\end{proof}

\end{document}